\documentclass[12pt,a4paper]{amsart}
\usepackage{a4wide}
\usepackage{amsmath, amssymb, amsfonts,enumerate}
\usepackage[all]{xy}
\usepackage{amscd}

\newtheorem{theorem}{Theorem}[section]

\newtheorem{corollary}[theorem]{Corollary}
\newtheorem{proposition}[theorem]{Proposition}

 \theoremstyle{definition}

\numberwithin{equation}{section}
\newcommand {\Z}{\mathbb{Z}} 
\newcommand {\R}{\mathbb{R}} 
\newcommand {\Q}{\mathbb{Q}} 

\newcommand{\T}{\mathbb{T}}

   \DeclareMathOperator{\GL}{GL}

\DeclareMathOperator{\M}{M}

\begin{document}
\title[A Garden of Eden theorem for Anosov diffeomorphisms]{A Garden of Eden theorem for Anosov diffeomorphisms on tori}
\author{Tullio Ceccherini-Silberstein}
\address{Dipartimento di Ingegneria, Universit\`a del Sannio, C.so
Garibaldi 107, 82100 Benevento, Italy}
\email{tceccher@mat.uniroma3.it}
\author{Michel Coornaert}
\address{Institut de Recherche Math\'ematique Avanc\'ee,
UMR 7501,                                             Universit\'e  de Strasbourg et CNRS,
                                                 7 rue Ren\'e-Descartes,
                                               67000 Strasbourg, France}
\email{coornaert@math.unistra.fr}
\begin{abstract}
Let $f$ be an Anosov diffeomorphism of the $n$-dimensional torus $\T^n$ 
and   $\tau$  a continuous self-mapping of $\T^n$  commuting with $f$.
We prove that $\tau$ is surjective   
if and only if the restriction of $\tau$ to each homoclinicity class of $f$ is injective.
\end{abstract}
\date{\today}
\maketitle

\section{Introduction}

The   Garden of Eden theorem,
originally established by Moore~\cite{moore} and Myhill~\cite{myhill} in the early 1960s, is an  important result in symbolic dynamics and coding theory. 
It provides a necessary and sufficient condition for a cellular automaton to be surjective.
More specifically, consider  a finite set $A$ and  the set $A^\Z$ consisting of all bi-infinite sequences $x = (x_i)$ with $x_i \in  A$ for all $i \in \Z$.
We equip $A^\Z$ with its \emph{prodiscrete topology}, that is, with the 
topology of pointwise convergence (this is also the product topology obtained by taking the discrete topology on each factor $A$ of $A^\Z$).
A \emph{cellular automaton} is a continuous map $\tau \colon A^\Z \to A^\Z$ that commutes with the shift homeomorphism $\sigma \colon A^\Z \to A^\Z$ 
given by $\sigma(x) = (x_{i - 1})$ for all  $x = (x_i) \in  A^\Z$.
Two sequences $x = (x_i),y  = (y_i) \in A^\Z$ are said to be \emph{almost equal} if one has 
$x_i = y_i$ for all but finitely many $i \in \Z$.
A cellular automaton $\tau \colon A^\Z \to A^\Z$ is called \emph{pre-injective}
if there exist no distinct sequences $x, y \in A^\Z$ that are almost equal and satisfy 
$\tau(x) = \tau(y)$.
The Moore-Myhill Garden of Eden theorem states that a cellular automaton 
$\tau \colon A^\Z \to A^\Z$ is surjective if and only if it is pre-injective.
The implication surjective $\Rightarrow$ pre-injective was first established by Moore~\cite{moore},
and Myhill~\cite{myhill} proved the converse implication shortly after.
\par
The Moore-Myhill Garden of Eden theorem has been extended in several directions.
There are now  versions of it for 
cellular automata over amenable groups \cite{machi-mignosi}, \cite{ceccherini},
cellular automata over subshifts \cite{gromov-esav}, \cite{fiorenzi-sofic}, 
\cite{Fiorenzi-strongly},
and linear cellular automata over linear shifts and subshifts \cite{csc-linear-goe}, 
\cite{cc-goe-lin-sub}
(the reader is refered to the monograph \cite{book} for a detailed exposition of some of these extensions, as well as historical comments and additional references).
\par
In this note, we  present an analogue of the Garden of Eden theorem for Anosov diffeomorphisms on tori.
This reveals one more connection between symbolic dynamics and the theory of smooth dynamical systems.
Actually our motivation came from a phrase of Gromov
\cite[p.~195]{gromov-esav}
which mentioned
the possibility of extending the Garden of Eden theorem to a suitable class of hyperbolic dynamical systems.
\par
Let $(X,f)$ be a dynamical system consisting of a compact metrizable space $X$ equipped with  a homeomorphism $f \colon X \to X$.
Two points in $X$ are called $f$-\emph{homoclinic} if their $f$-orbits are asymptotic both in the past and the future
(see Section~\ref{sec:background} for a precise definition). 
Homoclinicity  defines an equivalence relation on $X$.
 An \emph{endomorphism} of the dynamical system $(X,f)$ is
a continuous map $\tau \colon X \to X$ commuting with $f$.
We say that an endomorphism $\tau$ of $(X,f)$ is
\emph{pre-injective} (with respect to $f$) if
the restriction of $\tau$ to each $f$-homoclinicity class is injective (i.e., 
there is no pair of distinct $f$-homoclinic points  in $X$ having the same image under 
$\tau$)
(in the particular case when $X = A^\Z$ and $f = \sigma$ is the shift homeomorphism,
the endomorphisms of $(X,f)$ are precisely the cellular automata 
and this definition of pre-injectivity is equivalent to the one given above,
see e.g. \cite[Proposition~2.5]{csc-myhyp}). 
We  say that the dynamical system $(X,f)$ has the \emph{Moore property} if every surjective endomorphism of $(X,f)$ is pre-injective 
and that $(X,f)$ has the \emph{Myhill property}
if every pre-injective endomorphism of $(X,f)$  is surjective.
We say that the dynamical system $(X,f)$ has the \emph{Moore-Myhill property},
or that it satisfies the \emph{Garden of Eden theorem}, 
if $(X,f)$  has both the Moore and the Myhill properties.
\par
A $C^1$-diffeomorphism  $f$ of a compact $C^r$-differentiable ($r \geq 1$) manifold $M$ is called an \emph{Anosov diffeomorphism} if
the tangent bundle of $M$ splits as a direct sum $TM = E_s \oplus E_u$ 
of two invariant subbundles $E_s$ and $E_u$ such that, with respect to some (or equivalently any) Riemannian metric on $M$, the differential $df$  is uniformly contracting on $E_s$ and uniformly expanding on $E_u$ 
(see~\cite{smale}, \cite{brin-stuck}, \cite{dgs_ergodic-theory},  
\cite{kh-modern-theory-ds},  \cite{shub-global-stability}).
\par
Our main result is the following.

\begin{theorem}[Garden of Eden theorem for toral Anosov diffeomorphisms]
\label{t:anosov-torus}
Let $f$ be an Anosov diffeomorphism of the $n$-dimensional torus $\T^n$.
Then the dynamical system $(\T^n,f)$ has  the Moore-Myhill property.
In other words, if $\tau \colon \T^n \to \T^n$ is a continuous map commuting with $f$, then $\tau$ is surjective if and only if 
the restriction of $\tau$ to each homoclinicity class of $f$ is injective.
\end{theorem}

The paper is organized as follows.
In Section 2, we fix notation and present some background material on dynamical systems.
In Section 3, we establish Theorem~\ref{t:anosov-torus}.
The proof uses two classical results in the theory of hyperbolic dynamical systems.
The first one is the Franks-Manning theorem~\cite{franks}, \cite{manning},
which states that any Anosov diffeomorphism on $\T^n$ is topologically conjugate to a hyperbolic toral automorphism. The second is a theorem due
to Walters~\cite{walters}
which asserts that all endomorphisms of a hyperbolic toral endomorphism are affine.
This allows us to reduce the proof to an elementary question in linear algebra. 
In the final section, we discuss some examples and give an extension of the Myhill implication of the Garden of Eden theorem for topologically
mixing basic sets of Axiom A diffeomorphisms.

\section{Background}
\label{sec:background}

In this section, we review some basic facts about dynamical systems.  For more details, the reader is referred to the monographs \cite{brin-stuck}, \cite{dgs_ergodic-theory}, \cite{kh-modern-theory-ds}, \cite{lind-marcus}, and~\cite{shub-global-stability}. 

\subsection{Dynamical systems}
Throughout  this paper, by a \emph{dynamical system}, we mean a pair  $(X,f)$, where
$X$ is a compact metrizable space  and  
 $f \colon X \to X$ is  a homeomorphism.
 Sometimes, we shall simply write $f$ or $X$ instead of $(X,f)$ if there is no risk of confusion. 
We denote by $d$   a metric  on $X$ that is compatible with the topology.
\par
The \emph{orbit} of a point $x \in X$ is the set 
$\{f^n(x) : n \in \Z\} \subset X$. The point $x$ is called \emph{periodic} if its orbit is finite.
A subset $Y \subset X$ is said to be \emph{invariant} if $f(Y) = Y$.
If $Y \subset X$ is an invariant subset, we denote by $f\vert_Y$ the restriction of $f$ to 
$Y$, i.e., the map
$f\vert_Y \colon Y \to Y$ given by $f\vert_Y(y)  := f(y)$ for all $y \in Y$.
\par
One says that the dynamical systems $(X,f)$ and $(Y,g)$ are \emph{topologically conjugate} if there exists a homeomorphism $\varphi \colon X \to Y$ such that 
$\varphi \circ f = g \circ \varphi$.

\subsection{Homoclinicity}
Two points $x, y \in X$ are called \emph{homoclinic} with respect to 
$f$ (or $f$-\emph{homoclinic}) if one has
$d(f^n(x),f^n(y)) \to 0$ as $|n| \to \infty$.
Homoclinicity is an equivalence relation on $X$.
By compactness, 
this equivalence relation  is independent  of the choice of the metric $d$.

\begin{proposition}
\label{p:periodic-homoclinic}
Let $(X,f)$ be a dynamical system.
Suppose that  $x$ and $y$ are periodic points of $f$.
If $x$ and $y$ are $f$-homoclinic, then $x = y$.
\end{proposition}

\begin{proof}
Since $x$ and $y$ are periodic, there are integers $m, n \geq 1$ such that $f^n(x) = x$ and $f^m(y) = y$.
If $x$ and $y$ are $f$-homoclinic, then, given any $\varepsilon > 0$, 
we have $d(x,y) = d(f^{kmn}(x),d^{kmn}(y)) < \varepsilon$ for $k$ large enough.
This implies $x = y$. 
\end{proof}

\begin{proposition}
\label{p:inamge-homoclinic-equiv}
Let $(X,f)$ and $(Y,g)$ be two dynamical systems.
Suppose that $\psi \colon X \to Y$ is a continuous map such that $\psi \circ f = g \circ \psi$.
If two points in $X$ are $f$-homoclinic, then their images under $\psi$ are $g$-homoclinic. \end{proposition}

\begin{proof}
Suppose that the points $x,y \in X$ are $f$-homoclinic. 
Let $d$ (resp.~$d'$) be  a metric on $X$ (resp.~$Y$) that is compatible with the topology.
Then, we have that 
$$
d'(g^n(\psi(x)),g^n(\psi(y))) 
= d'(\psi(f^n(x)),\psi(f^n(y))) \to 0
$$
 as $|n| \to \infty$
since $d(f^n(x),f^n(y)) \to 0$ as 
$|n| \to \infty$ and $\psi$ is uniformly continuous.
This shows that the points $\psi(x)$ and $\psi(y)$ are $g$-homoclinic.
\end{proof}

\subsection{Hyperbolic toral automorphisms}
Consider the $n$-dimensional torus  $\T^n := \R^n/\Z^n$.
For $x \in \R^n$, we write $\overline{x} := x + \Z^n \in \T^n$.
\par
Let $\M_n(\Z)$ denote the ring of $n \times n$ matrices with integral entries.
Every matrix $A \in \M_n(\Z)$ induces a differentiable group endomorphism
 $f_A \colon \T^n \to \T^n$ given by 
$f_A(\overline{x}) = \overline{Ax}$ for all $x \in \R^n$.
One says that $f_A$ is the \emph{toral endomorphism} associated with $A$.
\par
The group of invertible elements of  $\M_n(\Z)$ is the group $\GL_n(\Z)$ of $n \times n$ matrices with integral entries and determinant $\pm 1$.
If $A \in \GL_n(\Z)$, then $f_A$ is a differentiable automorphism  of $\T^n$  and one says that   $f_A$ is the  
\emph{toral automorphism} associated with $A$.
If $f$ is a toral automorphism of $\T^n$,  the homoclinicity class of $\overline{0}$ is a subgroup of $\T^n$,  called the \emph{homoclinicity group} of  $f$, and denoted by $\Delta(f)$ (cf.~\cite{lind-schmidt}).
Note that two points  $p,q \in \T^n$ are $f$-homoclinic if and only if $p - q \in \Delta(f)$.
Observe  also that every point in $\Q^n/\Z^n$ is $f$-periodic so that
$\Delta(f) \cap \Q^n/\Z^n = \{\overline{0}\}$ by Proposition~\ref{p:periodic-homoclinic}.
\par
A matrix $A \in \GL_n(\Z)$ is called \emph{hyperbolic} if its complex spectrum does not meet the unit circle.
A diffeomorphism $f$ of $\T^n$ is called a
\emph{hyperbolic automorphism}  if there is a hyperbolic matrix  
 $A \in  \GL_n(\Z)$ such that $f = f_A$.
 \par
There are hyperbolic automorphisms on any torus of dimension $n \geq 2$
and every hyperbolic automorphism is Anosov.
On the other hand, by  the Franks-Manning theorem (cf.~\cite{franks} and \cite{manning}), 
every Anosov diffeomorphism of $\T^n$ is topologically conjugate to a hyperbolic automorphism. 

\section{Proof of the main result}
\label{sec:proof}

\begin{proof}[Proof of Theorem~\ref{t:anosov-torus}]
By the Franks-Manning theorem mentioned above, we can assume that $f$ is a hyperbolic automorphism of $\T^n$.
Let $A \in \GL_n(\Z)$ be a hyperbolic matrix such that $f = f_A$.
Let $\Delta(f) \subset \T^n$ denote the homoclinicity group of $f$.
\par
Let $\tau \colon \T^n \to \T^n$ be a continuous map commuting with $f$.
Since $f$ is a hyperbolic automorphism of $\T^n$ and $\tau$ commutes with $f$,
it follows from \cite[Theorem~2]{walters} that   $\tau$ is an affine toral endomorphism, that is,
there is a matrix $B \in \M_n(\Z)$ and $c \in \R^n$ such that
$\tau(\overline{x}) = \overline{Bx +c}$ for all $x \in \R^n$. 
Note that the map $\overline{x} \mapsto \tau(\overline{x}) - \overline{c}$ is a group endomorphism of $\T^n$.
\par
Suppose first that $\tau$ is surjective.
We claim that $\det(B) \not= 0$.
Indeed, otherwise, the image  of $\R^n$ under the affine map  $x \mapsto Bx+c$ would be an affine subspace
$L \subset \R^n$ with empty interior and  we would deduce from the Baire category theorem that
$L + \Z^n \subsetneqq \R^n$,  which  would contradict the surjectivity of $\tau$.
Thus, we have  $\det(B) \not= 0$ and hence $B \in \GL_n(\Q)$.
\par
Let $x, y \in \R^n$ such that the points $\overline{x}$ and $\overline{y}$ are $f$-homoclinic and satisfy
$\tau(\overline{x}) = \tau(\overline{y})$.
We then have   $B(x - y) \in \Z^n$ and hence  $x -y \in \Q^n$.
This implies that the point $\overline{x - y} $ is $f$-periodic.
On the other hand, since the points $\overline{x}$ and $\overline{y}$ are
$f$-homoclinic,  we have that  $\overline{x - y} = \overline{x} - \overline{y} \in \Delta(f)$.
By applying Proposition~\ref{p:periodic-homoclinic}, we deduce that $\overline{x} - \overline{y} = \overline{0}$ and therefore $\overline{x} =
\overline{y}$.
This shows that $\tau$ is pre-injective and hence that $(\T^n,f)$ has the Moore property.
\par
It remains to show that $(\T^n,f)$ has the Myhill property.
So, let us assume now that $\tau$ is pre-injective.
Since $f$ is a hyperbolic automorphism, it is known that the group $\Delta(f)$ is isomorphic to $\Z^n$
(see \cite[Example~3.3]{lind-schmidt}).
On the other hand, since $\tau(\overline{0}) = \overline{c}$,  we have that
$\tau(\Delta(f)) - \overline{c} \subset \Delta(f)$ by Proposition~\ref{p:inamge-homoclinic-equiv}.
As the restriction of $\tau$ to $\Delta(f)$ is injective by our pre-injectivity hypothesis, we deduce that
$\tau(\Delta(f)) - \overline{c}$ is a finite-index subgroup of $\Delta(f)$.
It is also known that $\Delta(f)$ is dense in $\T^n$ (see again \cite[Example~3.3]{lind-schmidt}).
Consider now  the closure $C \subset \T^n$  of $\tau(\Delta(f))  - \overline{c}$.
As $C$ is a closed subgroup of $\T^n$ and hence a torus, we must have $C = \T^n$ since otherwise the group $\Delta(f)$ would be contained in the
union of a finite number of translates of a torus of dimension less than $n$ and then could not be dense in $\T^n$.
It follows that the closure of $\tau(\Delta(f))$ is also equal to $\T^n$.
By continuity, this shows that $\tau$ is surjective.
 Consequently, $(\T^n,f)$ has the Myhill property.
\end{proof}

\newpage
\section{Concluding remarks}
\subsection{Examples of non-injective pre-injective endomorphisms}
\label{ss:injective-not-pre}
Injectivity trivially implies pre-injectivity
for endomorphisms of dynamical systems $(X,f)$.
However, the converse is false if $f$ is an Anosov diffeomorphism of $\T^n$.
Indeed, if $f$ is a hyperbolic automorphism of $\T^n$ and $m \in \Z$ satisfies $|m| \geq 2$,
then  multiplication by $m$ on $\T^n$ is an endomorphism of $(\T^n,f)$  that is surjective and hence pre-injective but not injective (its kernel has cardinality $|m|^n$). 
\par
This can be generalized in the following way. 
Let $f_i$ be a hyperbolic automorphism of the $n_i$-dimensional torus
$\T^{n_i}$, where $n_i \geq 2$ and $1 \leq i \leq k$.
Then $f := f_1 \times \dots \times f_k$ is a hyperbolic toral automorphism of the $N$-torus
$$
\T^N = \T^{n_1} \times \dots \times \T^{n_k},
$$
where $N := n_1 + \dots + n_k$.
Let us fix some non-zero integers  $m_i \in \Z$, $1\leq i \leq k$, with $|m_i| \geq 2$ for at least one $i$, and consider the endomorphism $\tau$ of $(\T^N,f)$ defined by
$$
\tau(x) := (m_1 x_1,\dots,m_k x_k)
$$
for all $x = (x_1,\dots,x_k) \in \T^N$.
Clearly $\tau$  is surjective and hence pre-injective. 
On the other hand, the kernel of $\tau$  has cardinality $|m_1|^{n_1}  \cdots  |m_k|^{n_k}$
and therefore $\tau$ is not injective.

\subsection{Ergodic toral automorphisms}
Let $A \in \GL_n(\Z)$ and $f_A \colon \T^n \to \T^n$ the associated toral automorphism. It is well known (see e.g.\ \cite[Proposition
24.1]{dgs_ergodic-theory}) that $f_A$ is ergodic (with respect to the Lebesgue measure on $\T^n$) if an only if $A$ has no eigenvalues which are
roots of unity.
This implies in particular that every hyperbolic toral automorphism of $\T^n$ is ergodic.
Observe that  the proof of the Moore property for hyperbolic toral automorphisms given in
Section~\ref{sec:proof}  applies verbatim to ergodic toral automorphisms.
Indeed, a continuous map $\tau \colon \T^n \to \T^n$ commuting with an ergodic toral automorphism $f_A$ is necessarily affine (cf.
\cite[Theorem~2]{walters}).
We thus have the following:

\begin{proposition}
Let $f \colon \T^n \to \T^n$ be an ergodic toral automorphism.
Then the dynamical system $(\T^n,f)$ has the Moore property.
In other words, if $\tau \colon \T^n \to \T^n$ is a surjective continuous map commuting with $f$, then the restriction of
$\tau$ to each $f$-homoclinicity class is injective. \qed
\end{proposition}

We know that hyperbolic toral automorphisms also satisfy the Myhill property by Theorem~\ref{t:anosov-torus}.
For ergodic toral endomorphism, however, the Myhill property fails to hold in general.
Consider for instance the matrix
\[
A = \begin{pmatrix}
0 & 0 & 0 & 1\\
-1 & 0 & 0 & 2\\
0 & -1 & 0 & 1\\
0 & 0 & -1 & 2
\end{pmatrix} \in \GL_4(\Z).
\]
Its eigenvalues are
\[
\begin{split}
\lambda_1 & = \frac{1}{2} - \frac{1}{\sqrt{2}} + i\frac{\sqrt{\sqrt{8}+1}}{2}\\
\lambda_2 & = \frac{1}{2} - \frac{1}{\sqrt{2}} - i\frac{\sqrt{\sqrt{8}+1}}{2}\\
\lambda_3 & = \frac{1}{2} + \frac{1}{\sqrt{2}} - \frac{\sqrt{\sqrt{8}-1}}{2}\\
\lambda_4 & = \frac{1}{2} + \frac{1}{\sqrt{2}} + \frac{\sqrt{\sqrt{8}-1}}{2}
\end{split}
\]
and satisfy $|\lambda_1| = |\lambda_2| = 1$ and $0 < \lambda_3 < 1 < \lambda_4$.
Since none of these eigenvalues is a root of unity, the associated toral automorphism $f_A \colon \T^4 \to \T^4$ is ergodic. On the other hand,
the characteristic polynomial $\chi_A(x) = x^4 - 2x^3 + x^2 - 2x +1$ is irreducible over $\Q$.
It follows that  the homoclinicity group $\Delta(f_A)$ is reduced to $0$ (cf. \cite[Example 3.4]{lind-schmidt}).
Consequently, every endomorphism of the dynamical system $(\T^4,f_A)$ is pre-injective with respect to $f_A$.
Since the zero endomorphism is not surjective, we conclude that $(\T^4,f_A)$ does not have the Myhill property.

\subsection{The Myhill property for elementary basic sets}
Let us first recall some definitions.
\par
Let $f$ be a homeomorphism of a compact metrizable space $X$. 
One says that the dynamical system $(X,f)$  is \emph{expansive} if there exists a constant 
$\delta  > 0$ such that, for every pair of distinct points $x,y \in X$, there exists 
$n = n(x,y)  \in \Z$ such that
$d(f^n(x),f^n(y)) \geq  \delta$ (here $d$ denotes any metric on $X$ that is compatible with the topology).
One says that the dynamical system $(X,f)$ is \emph{topologically mixing} if
for any pair of non-empty open subsets 
$U,V \subset X$, there is  an integer $N = N(U,V) \in \Z$ such that $f^n(U)$ meets $V$ for all   $n \geq N$.  
\par
One says that the dynamical system $(X,f)$ is a \emph{factor} of the dynamical system 
$(Y,g)$ if there exists a continuous surjective map
$\pi \colon Y \to X$ such that $\pi \circ g = f \circ \pi$.
Such a map $\pi$ is then called a \emph{factor map}.
A factor map $\pi \colon Y \to X$ is said to be \emph{uniformly bounded-to-one} if there is an integer $K \geq 1$ such that each $x \in X$
has at most $K$ pre-images in $Y$.
\par
Let $A$ be a finite set and let $\sigma$ denote the shift homeomorphism on $A^\Z$.
A  $\sigma$-invariant closed subset $\Sigma \subset A^\Z$ is called a
\emph{subshift}.
A subshift $\Sigma \subset A^\Z$ is said to be \emph{of finite type}
if there is an integer $n \geq 1$ and a subset $P \subset A^n$ such that $X$ consists of the  sequences $x = (x_i) \in A^\Z$ that satisfy
$$
(x_i, x_{i + 1}, \dots , x_{i + n - 1}) \in P
$$
for all $i \in \Z$.
\par
We have the following result.

\begin{theorem}
\label{t:myhill-property}
Let $f \colon X \to X$ be a homeomorphism of a compact metrizable space $X$. 
Suppose that  the dynamical system $(X,f)$ is expansive 
and that there exist a finite set $A$, a topologically mixing subshift of finite type
$\Sigma \subset A^\Z$, and a uniformly bounded-to-one factor map 
$\pi \colon \Sigma \to X$.
Then the dynamical system $(X,f)$ has the Myhill property.
\end{theorem}

\begin{proof}
This is a special case of \cite[Theorem~1.1]{csc-myhyp}
since the group $\Z$ is amenable and every topologically mixing subshift of finite type over $\Z$ is strongly irreducible.
\end{proof}

Note that there exist dynamical systems $(X,f)$
satisfying all the hypotheses of Theorem~\ref{t:myhill-property}
that do not have the Moore property.
An example of such a dynamical system is provided by the \emph{even subshift}, that is, the subshift $X \subset \{0,1\}^\Z$ consisting of all bi-infinite sequences of $0$s and $1$s with an even number of $0$s between any two $1$s.
Indeed, if $\Sigma \subset \{0,1\}^\Z$ denotes the \emph{golden subshift}, that is, the subshift consisting of all bi-infinite sequences of $0$s and $1$s with no consecutive $1$s,  
it is known that $\Sigma$ is a topologically mixing subshift of finite type  and 
  that there is a factor map $\pi \colon \Sigma \to X$ such that   each configuration in $X$ has at most  $2$ pre-images in $\Sigma$
(see e.g. \cite[Example~4.1.6]{lind-marcus}).
Thus, the even subshift satisfies the hypotheses of Theorem~\ref{t:myhill-property}. 
On the other hand, Fiorenzi \cite[Section~3]{fiorenzi-sofic} 
proved that the even subshift does not have the Moore property.
\par
Let $M$ be a compact $C^r$-differentiable ($r \geq 1$) manifold and $f$ a $C^1$-diffeomorphism of $M$.
One says that a closed $f$-invariant subset $\Lambda \subset M$ is a \emph{hyperbolic set}
if the restriction to $\Lambda$ of the 
tangent bundle of $M$ continuously splits as a direct sum of two invariant subbundles $E_s$ and $E_u$ such that, with respect to some (or equivalently any) Riemannian metric on $M$, the differential $df$ of $f$ is uniformly contracting on $E_s$ and uniformly expanding on $E_u$, i.e., there are constants 
$C > 0$ and $0 < \lambda < 1$ such that
$\Vert df^n(v) \Vert \leq C\lambda^n \Vert v \Vert$
and
$\Vert df^{-n}(w) \Vert \leq C \lambda^n \Vert w \Vert$
for all $x \in \Lambda$, $v \in E_s(x)$,  $w \in E_u(x)$, and $n \geq 0$.
Thus, $f$ is an Anosov diffeomorphism if and only if the whole manifold $M$ is a hyperbolic set for $f$.
A point $x \in M$ is called \emph{non-wandering} if for every neighborhood $U$ of $x$, there is an integer $n \geq 1$ such that $f^n(U)$ meets $U$.
The set $\Omega(f)$ consisting of all non-wandering points of $f$ is a closed invariant subset of $M$.
One says that $f$ satisfies Smale's  \emph{Axiom A}  if the set $\Omega(f)$  is hyperbolic and the periodic points of $f$ are dense in $\Omega(f)$
(cf.~\cite{smale}).
If $f$ is Axiom A, then $\Omega(f)$ can be uniquely written as a disjoint union of closed invariant subsets $\Omega(f) = X_1 \cup \dots \cup X_k$, such that the restriction of $f$ to  each $X_i$ is topologically transitive for $1 \leq i \leq k$ (spectral decomposition theorem).
These subsets $X_i$ are called the \emph{basic sets} of $(M,f)$.
A basic set $X_i$ is called \emph{elementary} if the restriction of $f$ to $X_i$ is topologically mixing.

\begin{corollary}
\label{cor:goe-hyperbolic-set}
Let $f$ be a $C^1$-diffeomorphism of a compact $C^r$-differentiable ($r \geq 1$) manifold $M$
that satisfies Axiom A.
Suppose  that $X$ is an elementary basic set of $(M,f)$
and let $f\vert_X \colon X \to X$ denote the restriction of $f$ to $X$.
Then the dynamical system $(X,f\vert_X)$ has the Myhill property. 
\end{corollary}

\begin{proof}
The fact that the dynamical system $(X,f\vert_X)$ satisfies the hypothesis of 
Theorem~\ref{t:myhill-property} 
follows from the classical work of Rufus Bowen on Axiom A diffeomorphisms.
The expansivity of $f\vert_X$ is shown in \cite[Lemma~3]{bowen-markov-1970}.
On the other hand, Bowen used a Markov partition to show that
 one can find  a finite set $A$ (the set of rectangles of the Markov partition) and    a topologically mixing subshift of finite type 
 $\Sigma \subset A^\Z$
such that there exists  a uniformly bounded-to-one factor map $\pi \colon \Sigma \to X$
(cf. \cite[Theorem~28 and Proposition~30]{bowen-markov-1970} and 
\cite[Proposition~10]{bowen-markov-minimal_AJM-1970}).
\end{proof}

It is an open question whether every Anosov diffeomorphism  is topologically mixing. 
However, for an Anosov diffeomorphism $f \colon M \to M$,    the following conditions are known to be equivalent 
(see e.g.~\cite[Theorem 5.10.3]{brin-stuck}):
(a) $f$ is topologically mixing;
(b) $f$ is topologically transitive;
(c) every point in $M$ is non-wandering; 
(d) the periodic points of $f$ are dense in $M$.
This implies in particular  that every topologically mixing Anosov diffeomorphism $f$ of a compact manifold 
$M$ is Axiom A with $\Omega(f) = M$. Thus,
as a particular case of Corollary~\ref{cor:goe-hyperbolic-set}, we get the following result, which
extends the Myhill part of Theorem~\ref{t:anosov-torus} to all topologically mixing Anosov diffeomorphisms.

\begin{corollary}
\label{cor:goe-anosov}
If $f$ is a topologically mixing Anosov diffeomorphism of a compact $C^r$-differentiable ($r \geq 1$) manifold $M$,
then the dynamical system $(M,f)$ has the Myhill property.
\qed
\end{corollary}

\subsection{Zero-dimensional basic sets}
If $X$ is a zero-dimensional basic set of an Axiom A diffeomorphism $f$, it follows from \cite[Theorem~6.6]{bowen-top-entropy-axiomA} that $(X,f\vert_X)$ is topologically conjugate to an irreducible subshift of finite type.
As every irreducible subshift of finite type has the Moore-Myhill property by 
\cite[Corollary~2.19]{fiorenzi-sofic},
we deduce that $(X,f\vert_X)$ has the Moore-Myhill property.
\par
In view of   this observation and of   Theorem~\ref{t:anosov-torus},
it is very tempting to conjecture that 
the restriction of an Axiom A diffeomorphism  to a (possibly non-elementary) basic set  always has the 
Moore-Myhill property.
 
\bibliographystyle{siam}

\end{document}